\documentclass[12pt]{amsart}

\usepackage{amsmath,amssymb,amsthm,hyperref,cite,color, booktabs, geometry, tikz-cd, mathtools, cleveref}

\newtheorem{Theorem}{Theorem}[section]

\newtheorem{Lemma}[Theorem]{Lemma}

\newtheorem{Corollary}[Theorem]{Corollary}

\newcommand\Res{\mathrm{Res}}
\renewcommand\char{\mathrm{char}}
\newcommand\Br{\mathrm{Br}}
\newcommand\N{\mathrm{N}}
\newcommand\C{\mathrm{C}}
\newcommand\Ind{\mathrm{Ind}}
\newcommand\core{\mathrm{core}}
\renewcommand\ker{\mathrm{ker}}
\newcommand\Hom{\mathrm{Hom}}

\newcommand\soc{\mathrm{soc}}
\renewcommand\O{\mathrm{O}}

\title{Kernel of Scott modules and Brauer indecomposability}

\date{}

\author{Lin Wu}
\address{Department of Mathematics, Southern University of Science and Technology, Shenzhen 518055, China}
\email{12431015@mail.sustech.edu.cn}

\begin{document}

\maketitle

\begin{abstract}
Let $k$ be an algebraically closed field of prime characteristic $p$. Let $G$ be a finite group. We investigate the Brauer indecomposability of Scott $kG$-modules in relation to the kernel of modules. We generalize a criterion for Brauer indecomposability. We also prove that, in certain cases, Brauer indecomposability of a Scott $kG$-module can be lifted from that of a Scott module over a $p$-local subgroup.
\end{abstract}

\section{Introduction}

Let $G$ be a finite  group and let $p$ be a prime number. Let $k$ be an algebraically closed field with $\char(k)= p$. Let $P\leq G$ be a $p$-subgroup. For any finite dimensional $kG$-module $M$, the Brauer construction $\Br_P(M)= M(P)$ is a $k\N_G(P)$-module (for backgroud of Brauer construction, we refer the reader to \cite{Linckelmann_2018}). If for any $p$-subgroup $Q\leq G$, $\Res_{Q\C_G(Q)}^{\N_G(Q)}M(Q)$ is indecomposable or equals to $0$, we say $M$ is Brauer indecomposable. 

In 2011, Kessar, Kunugi and Mitsuhashi introduced the concept of Bruaer indecomposability in  \cite{KESSAR201190}. They gave this definition for the purpose of studying gluing processes, which is a method for studying categorical equivalence between $p$-blocks of finite groups (see \cite{KOSHITANI2005726, MR2078933}). Brou\'e 's gluing method \cite[Theorem 6.3]{MR1308978} gives us a sufficient condition for $S(G\times H, \Delta P)$ to induce a stable equivalence of Morita type between two principal blocks. If this is the case, we always have that $S(G\times H, \Delta P)$ is Brauer indecomposable, and no exceptions are known so far. For reference, one could see Koshitani and  Lassueur's works \cite{MR4050079, MR4102107, MR4335855}. So it is important to know whether the Scott module is Brauer indecomposable or not.

Regarding the Brauer indecomposability of Scott modules, one may refer to \cite{KESSAR201190, ISHIOKA2017441, doi:10.1142/S021949882550015X}. Our present paper primarily builds upon Ishioka and Kunugi's work \cite{ISHIOKA2017441}, Koshitani and Tuvay's work \cite{doi:10.1142/S021949882550015X, koshitani2025liftingbrauerindecomposabilityscott} to establish a series of relationships between the kernel of a Scott module and its Brauer indecomposability.

For subgroups $Q, R\leq G$, we let
\begin{equation*}
\Hom_G(Q, R)= \{\phi:Q\rightarrow R\mid \text{$\phi$ is induced by a conjugation}\}.
\end{equation*}
For a $p$-subgroup $P\leq G$, the fusion system $\mathcal{F}_P(G)$ is the category whose objects are the subgroups of $P$ and whose morphism set from $Q$ to $R$ is $\Hom_G(Q, R)$. 

In \cite{MR4034788, MR4277756, MR4298846}, Koshitani and Tuvay study the Brauer indecomposability of Scott modules with an explict vertex $P$, where they make use of \cite[Theorem 1.3]{ISHIOKA2017441},  and we introduce this theorem here.

\begin{Theorem}
Let $G$ be a finite group and $P$ a $p$-subgroup of $G$. Suppose that $M= S(G, P)$ and that $\mathcal{F}_P(G)$ is saturated. TFAE:

(1) $M$ is Brauer indecomposable.

(2) $\Res_{Q\C_G(Q)}^{\N_G(Q)}S(\N_G(Q), \N_P(Q))$ is indecomposable for each fully normalized subgroup $Q\leq P$.
\end{Theorem}

For a $p$-subgroup $P$ of a finite group $G$, let $\core_G(P):= \cap_{x\in G}{^xP}$ be the largest normal subgroup of $G$ contained in $P$ (when there is no ambiguity, we denote $\core_G(P)$ by $\core(P)$), where $^xP= xPx^{-1}$. For a $kG$-module $M$, we let 
\begin{equation*}
\ker(M)= \{g\in G\mid \text{$gm= m$ for all $m\in M$}\}.
\end{equation*}

From the conclusions in \cite{doi:10.1142/S021949882550015X, ISHIOKA2017441}, we observe the following fact: in certain cases, when we want to determine the Brauer indecomposability of the Scott module $S(G, P)$, the subgroups of $\core(P)$ are not important; what matters are the subgroups between $\core(P)$ and $P$. The following two theorems are intended to illustrate this fact.

There is a generalization of \cite[Theorem 1.3]{ISHIOKA2017441}.

\begin{Theorem}\label{Th2}
Let $G$ be a finite group and $P$ a $p$-subgroup of $G$. Let $M= S(G, P)$ and suppose that $\mathcal{F}_P(G)$ is saturated. TFAE:

(1) $M$ is Brauer indecomposable.

(2) $\Res_{Q\C_G(Q)}^{\N_G(Q)}S(\N_G(Q), \N_P(Q))$ is indecomposable for each fully normalized subgroup $Q$ satisfying $\core(P)\leq Q\leq P$.
\end{Theorem}

We introduce a reformulation of \cite[Corollary 1.2]{doi:10.1142/S021949882550015X}.

\begin{Theorem}\label{Th1}
Let $G$ be a finite group and $P$ a $p$-subgroup. Let $M= S(G, P)$. We assume $\mathcal{F}_P(G)$ is saturated. Let $R\lhd G$ be a normal subgroup such that $R\leq P\cap \ker(M)$. Let
\begin{equation*}
\mathcal{A}= \{T\leq P\mid \text{$R\leq T$ and $T$ is a maximal subgroup of $P$}\}.
\end{equation*}
If for any $T\in \mathcal{A}$ we have $\Res_{C_G(T)}^GM$ is indecomposable, then $M$ is Brauer indecomposable. In particular, $\core(P)\leq P\cap \ker(M)$, and any subgroup $R$ satisfying the above conditions is contained in $\core(P)$.
\end{Theorem}

We will see that $\core(P)\leq \ker(S(G, P))$ is always true. Now we impose an additional condition on the kernel of Scott module: $P\leq \ker(S(G, P))$. 

Koshitani and Tuvay investigate the lifting property of Brauer indecomposability of Scott modules from a subgroup to the whole group in their preprint \cite{koshitani2025liftingbrauerindecomposabilityscott}. We use the method of \cite[Theorem 1.1]{koshitani2025liftingbrauerindecomposabilityscott}, but change the conditions, and finally obtain the following result:

\begin{Theorem}\label{Th3}
Let $G$ be a finite group and let $P$ be a $p$-subgroup. Assume $P\leq \ker(S(G, P))$ and $\N_G(P)$ is a $p$-group. If $S(\N_G(P), P)$ is Brauer indecomposable, then $S(G, P)$ is Brauer indecomposable. 
\end{Theorem}

\section{Preliminaries}

Let $A$ be a ring with identity, and let $V, M$ be two $A$-modules. We use $V\mid M$ to denote $V$ is isomorphic to a direct summand of $M$.

Let $G$ be a finite group, let $P$ be a $p$-subgroup of $G$, let $Q\leq P$ and let $M$ be a $kG$-module. Let
\begin{equation*}
 M^P= \{m\in M\mid \text{$ym= m$ for all $y\in P$}\},
\end{equation*}
and let
\begin{equation*}
M_Q^P= \{\sum_{x\in [P/ Q]}xm\mid \text{$m\in M^Q$}\}.
\end{equation*}
We set $M(P)= M^P/ (\sum_{Q< P}M_Q^P)$ and denote by $\Br_P: M^P\rightarrow M(P)$ the canonical surjective map. We consider $M^P$ as an $k\N_G(P)$-module and $M(P)$ as a $k\N_G(P)$-module or as a $kN_G(P)/ P$-module. We will also use $\Br_P(M)$ to denote $M(P)$.

Let $G$ be a finite group and let $H_1, H_2$ be two subgroup. We use $H_1\leq_G H_2$ to denote there exists $g\in G$ such that $^gH_1\leq H_2$, and we use $H_1\not\leq_G H_2$ to denote the contrary. $H_1=_G H_2$ means $H_1\leq_G H_2$ and $H_2\leq_G H_1$ hold at the same time.
In the following lemma we will see that, when we want to prove $S(G, P)$ is Brauer indecomposable, we only need to consider subgroups of $P$. We will henceforth use this fact directly without further reference.
\begin{Lemma}
Let $G$ be a finite group and $P$ a $p$-subgroup. For any $p$-subgroup $Q\leq G$ such that $Q\not\leq_G P$, we have $\Br_Q(S(G, P))= 0$. In particular, if for any $Q\leq P$, $\Res_{Q\C_G(Q)}^{\N_G(Q)}\Br_Q(S(G, P))$ is indecomposable, then $S(G, P)$ is Brauer indecomposable.
\end{Lemma}
\begin{proof}
\cite[Proposition 5.10.3]{Linckelmann_2018}.
\end{proof}

Here is an easy observation on the kernel of $\Ind_P^G(k)$, where $k$ is the $1$-dimensional $kP$-module with trivial $P$-action.

\begin{Lemma}\label{Lem}
Let $G$ be a finite group and $H\leq G$. Let $M$ be a $kH$-module. Then we have $\ker(\Ind_H^G(M))= \core(\ker(M))$. In particular, $\ker(\Ind_H^G(k))= \core(H)$.
\end{Lemma}
\begin{proof}
\cite[p.170, Lemma 1.2]{nagao2014representations}.
\end{proof}

In fact, we can describe the kernel of $S(G, P)$ more clearly.

\begin{Lemma}\label{Lem2.4}
Let $G$ be a finite group and $P$ a $p$-subgroup. Let $K= \ker(S(G, P))$. Then $\O_p(K)= \core_G(P)$ and $\O_{p'}(K)= \O_{p'}(G)$.
\end{Lemma}
\begin{proof}
We already know $\core_G(P)\leq \O_p(K)$. Since $\O_p(K)$ acts trivially on $S(G, P)$, we get $\O_p(K)\leq P$ by \cite[p.294, Theorem 7.8 (1)]{nagao2014representations}. So $\O_p(K)\leq core_G(P)$.

Let $N= \O_{p'}(G)$. By \cite[p.297, Corollary 8.5]{nagao2014representations}, we have $S(G, P)\cong S(G, PN)$. Since $N$ acts trivially on $S(G, PN)$, we get $N\leq \O_{p'}(K)$. Conversely, since $\O_{p'}(K)\lhd G$, we get $\O_{p'}(K)= N$.
\end{proof}

For backgroud of fusion system, one could see \cite{Aschbacher_Kessar_Oliver_2011}. For the definition and basic properties of Scott module, one could see \cite{nagao2014representations}. For the convenience of the reader, we recall the definition and a basic property of Scott module. We use ${_Gk}$ to denote the 1-dimensional $kG$-module with trivial $G$-action.

Let $G\geq H$ and $Q\in Syl_p(H)$. (1) There exists an indecomposable component $S$ of $\Ind_H^G({_Hk})$ that satisfies the following three conditions: (a) ${_Gk}\mid \soc(S)$. (b) ${_Gk}\mid S/J(kG)S$. (c) $vx(S)=_G Q$. And if $f= f(G, Q, N= N_G(Q))$ denotes the Green correspondence, then $f(S)$ can be considered as an $k[N/Q]$-module in the canonical way, which is a projective cover of ${_{N/Q}k}$.

(2) For any indecomposable decomposition of $\Ind_H^G({_Hk})$, there exists a unique indecomposable component $S$ of $\Ind_H^G({_Hk})$ that satisfies any one of the three conditions above. This component is called the Scott module and is denoted by $S(G, H)$.

Let $Q\leq H\leq G$ and let $Q$ be a $p$-subgroup. Let $S= S(G, Q)$ and $S_1= S(H, Q)$. Then the following statements hold. (1) $S\mid \Ind_H^G(S_1)$. (2) $S_1\mid \Res_H^G(S)$.

\section{Proof of Theorems}

\begin{proof}[Proof of Theorem \ref{Th2}]
Since $(1)\Rightarrow (2)$ is already proved in \cite[Theorem 1.3]{ISHIOKA2017441}, we only prove $(2)\Rightarrow (1)$. Let $Q$ be a subgroup of $P$. We prove $\Res_{Q\C_G(Q)}^{\N_G(Q)}M(Q)$ is indecomposable by induction on $[P: Q]$. 

If $[P: Q]= 1$, we are done by \cite[Lemma 4.3]{KESSAR201190}. 

Now we assume $[P: Q]> 1$. Let $R= \core(P)$. By \cite[Proposition 2.6]{doi:10.1142/S021949882550015X}, it suffices to prove $\Res_{\C_G(QR)}^{\N_G(QR)}M(QR)$ is indecomposable. If $QR= P$, then we are done by \cite[Lemma 2.4]{ISHIOKA2017441}. Now we assume $QR< P$. Since $\mathcal{F}_P(G)$ is saturated, there exists an element $g\in G$ such that $^g(QR)$ is fully normalized in $\mathcal{F}_P(G)$. Let $L= {^g(QR)}$. Note that $\Res_{\C_G(L)}^{\N_G(L)}M(L)$ is indecomposable implies $\Res_{\C_G(QR)}^{\N_G(QR)}M(QR)$ is indecomposable. So we may assume $Q$ is fully normalized and $P> Q\geq R$.

The rest of the proof is the same as that in \cite[Theorem 1.3]{ISHIOKA2017441}. For the convenience of the reader, we include the proof here.

By \cite[Lemma 3.1, Theorem 2.2]{ISHIOKA2017441}, we get $S(\N_G(Q), \N_P(Q))\mid \Res_{\N_G(Q)}^GM$. Since $Q$ acts trivially on $S(\N_G(Q), \N_P(Q))$, we have that $S(\N_G(Q), \N_P(Q))\mid M(Q)$ by \cite[Lemma 2.1]{ISHIOKA2017441}.

There is an indecomposable decomposition $M(Q)= \oplus_{i= 1}^r N_i$ with $N_1\cong S(\N_G(Q), \N_P(Q))$. Suppose $r> 1$. Then for $i> 1$ we have $N_i\mid \Ind_{^tP\cap \N_G(Q)}^{\N_G(Q)}(k)$ for some $t\in G$ by Mackey's decomposition formula, where we use $k$ to denote the $1$-dimensional $k[\N_G(Q)\cap {^tP}]$-module with trivial $\N_G(Q)\cap {^tP}$-action. By \cite[Lemma 2.1, Lemma 2.6]{ISHIOKA2017441}, there is a vertex $T$ of $N_i$ such that $Q< T\leq \N_G(Q)\cap {^tP}$. Since $N_i$ is a trivial source module, we get $N_i(T)\neq 0$ by \cite[Proposition 5.10.3]{Linckelmann_2018}. By \cite[Lemma 3.2]{ISHIOKA2017441}, $\N_{^tP}(Q)\leq_{\N_G(Q)}\N_P(Q)$. Thus
\begin{equation*}
T\leq {^tP}\cap \N_G(Q)\leq_{\N_G(Q)}\N_P(Q).
\end{equation*}
So $N_1(T)\neq 0$. Now $N_1(T)\oplus N_i(T)\mid \Res_{\N_G(T)\cap \N_G(Q)}^{\N_G(T)}(M(T))$. Since $Q\lhd T$, we get $T\C_G(T)\leq \N_G(T)\cap \N_G(Q)$. So $\Res_{T\C_G(T)}^{\N_G(T)}(M(T))$ is decomposable. Since $[P: Q]> [P: {^{t^{-1}}T}]$, we get a contradiction by the induction hypothesis. Now we must have $r= 1$. So $M(Q)\cong S(\N_G(Q), \N_P(Q))$. By our assumption, $Res_{Q\C_G(Q)}^{\N_G(Q)}(M(Q))$ is indecomposable.
\end{proof}

\begin{proof}[Proof of Theorem \ref{Th1}]
If $R= P$, by \cite[Theorem 1.1]{doi:10.1142/S021949882550015X}, it suffices to prove $\Res_{\C_G(P)}^{\N_G(P)}M(P)$ is indecomposable, which is true by \cite[Lemma 4.3]{KESSAR201190}. Now we assume $R< P$. Let $Q$ be a subgroup such that $Q\leq P$. By \cite[Proposition 2.6]{doi:10.1142/S021949882550015X}, we may assume $R\leq Q< P$. There exists a $T\in \mathcal{A}$ such that $Q\leq T$. Now it suffices to prove that $\Res_{\C_G(Q)}^{\N_G(Q)}M(Q)$ is indecomposable. By \cite[5.8.5]{Linckelmann_2018}, $M(R)(Q)\cong M(Q)$ as $k\N_{\N_G(R)}(Q)$-modules. By \cite[Lemma 2.1]{ISHIOKA2017441}, $M(R)(Q)\mid \Res_{\N_{\N_G(R)}(Q)}^{\N_G(R)}M(R)$. Since $\C_G(Q)\leq \N_{\N_G(R)}(Q)$, we get $\Res_{\C_G(Q)}^{\N_G(Q)}M(Q)\mid Res_{\C_G(Q)}^GM$. By our assumption, we have $\Res_{\C_G(Q)}^GM$ is indecomposable. Thus $\Res_{\C_G(Q)}^{\N_G(Q)}M(Q)$ is also indecomposable. 

For the last statement, only need to see $\ker(M)\geq \ker(\Ind_P^G(k))= \core(P)$ by Lemma \ref{Lem}.
\end{proof}

\begin{proof}[Proof of Theorem \ref{Th3}]
Let $H= \N_G(P)$. Claim: $\Res_H^GS(G, P)$ is indecomposable.  By Mackey's decomposition formula, $\Res_{H}^GS(G, P)\mid \oplus_{t\in [H\backslash G/ P]} \Ind_{H\cap {^tP}}^H(k)$. Let $\Res_H^GS(G, P)= \oplus_{i= 1}^mN_i$ be an indecomposable decomposition. Since $P$ acts on $S(G, P)$ trivially, for any $1\leq i\leq m$, we get that $P$ is a vertex of $N_i$ by \cite[Lemma 2.1]{ISHIOKA2017441}. There exists $t\in G$ such that $N_i\mid \Ind_{H\cap {^tP}}^Hk$, then we get $P= H\cap {^tP}$. Thus $t\in H$ and $N_i\mid \Ind_P^Hk$. So we have
\begin{equation*}
S(H, P)\mid \Res_H^GS(G, P)\mid \Ind_P^Hk.
\end{equation*}
Since $P\lhd H$ with $p$-power index, $\Ind_P^Hk$ is an indecomposable $kH$-module by Green's indecomposability theorem. We complete the proof of our claim.

The rest follows from the proof of \cite[Theorem 1.1]{koshitani2025liftingbrauerindecomposabilityscott}. Let $Q$ be a subgroup of $P$. Since $S(H, P)$ is Brauer indecomposable by our assumption, we get $\Res_{\N_H(Q)}^HS(H, P)$ is indecomposable by \cite[Lemma 2.1]{ISHIOKA2017441}. $\Res_H^GS(G, P)$ is indecomposable implies that
\begin{equation*}
\Res_{\N_H(Q)}^GS(G, P)=\Res_{\N_H(Q)}^H\Res_H^GS(G, P)\cong \Res_{\N_H(Q)}^HS(H, P).
\end{equation*}
So $\Res_{\N_G(Q)}^GS(G, P)$ is indecomposable and thus $\Res_{\N_G(Q)}^GS(G, P)\cong \Br_Q(S(G, P))$. Note that 
\begin{equation*}
\Res_{\C_H(Q)}^{\N_G(Q)}\Br_Q(S(G, P))\cong \Res_{\C_H(Q)}^HS(H, P)
\end{equation*} 
is also indecomposable. Since $\C_H(Q)\leq \C_G(Q)$, we get that $\Res_{\C_G(Q)}^{\N_G(Q)}\Br_Q(S(G, P))$ is indecomposable, which implies $S(G, P)$ is Brauer indecomposable.  
\end{proof}

\begin{Corollary}\label{Coro}
Let $G$ be a finite group and let $P$ be a $p$-subgroup. Assume $P\leq \ker(S(G, P))$ and $G$ is $p$-nilpotent. If $S(\N_G(P), P)$ is Brauer indecomposable, then $S(G, P)$ is Brauer indecomposable.
\end{Corollary}

Remark. If $P= 1$ and $\mathcal{F}_P(G)$ is saturated, then $S(G, 1)$ is Brauer indecomposable by \cite[Theorem 1.2]{KESSAR201190}.

\begin{proof}[Proof of Corollay \ref{Coro}]
Let $H= \N_G(P)$. If $\Res_H^GS(G, P)$ is indecomposable, then the rest follows from the second paragragh of the proof of Theorem\ref{Th3}. So it suffices to prove $\Res_H^GS(G, P)$ is indecomposable.

By the proof of Theorem \ref{Th3}, we know $\Res_H^GS(G, P)\mid \Ind_P^Hk$ since $P\leq\ker(S(G, P))$. Let $X= \Res_H^GS(G, P)$. There are $kH$-homomorphisms $i, \pi$
\begin{equation*}
\begin{tikzcd}
X \arrow[r, "i"] & \Ind_P^H(k) \arrow[r, "\pi"] & X
\end{tikzcd}
\end{equation*}
such that $\pi i= id$. Let $L= H\cap \O_{p'}(G)$. By Lemma \ref{Lem2.4}, $L$ acts trivially on $X$. Let $R$ be an abelian subgroup of $\Ind_P^H(k)$ generated by $\{ax- x\mid a\in L, x\in \Ind_P^H(k)\}$. It is easy to see $R$ is a $kH$-submodule of $\Ind_P^H(k)$. For any $a\in L, x\in \Ind_P^H(k)$, we have $\pi(ax- x)= a\pi(x)- \pi(x)= \pi(x)- \pi(x)= 0$. So $X$ is a direct summand of $\Ind_P^H(k)/ R$. 

Claim: $\Ind_P^H(k)/ R\cong \Ind_{PL}^H(k)$. Let $[H/ LP]= \{h_1, h_2, \cdots, h_r\}$ and $[LP/ P]= \{c_1, c_2, \cdots, c_l\}$ be complete sets of left coset representatives. If $l= 1$, we get $PL= P$, then we are done. Now we assume $l> 1$. Let
\begin{equation*}
\Phi: \Ind_P^H(k)\rightarrow \Ind_{PL}^H(k), \sum_{i, j} h_ic_j\otimes\lambda_{ij}\mapsto \sum_{i, j} h_ic_j\otimes\lambda_{ij}
\end{equation*}
be a map. $\Phi$ is a $kH$-epimorphism. We only need to prove $\ker(\Phi)= R$. Since $L$ acts on $\Ind_{PL}^H(k)$ trivially, we get $R\leq \ker(\Phi)$. Now we assume $\Phi(\sum_{i, j} h_ic_j\otimes\lambda_{ij})= 0$, then $\sum_{i, j} h_i\otimes\lambda_{ij}= \sum_i h_i\otimes\sum_j(\lambda_{ij})= 0$. So for any $i_0\in \{1,\cdots, r\}$, $\sum_{j= 1}^l\lambda_{i_0j}= 0$. Thus $\sum_j h_{i_0}c_j\otimes\lambda_{i_0j}= \sum_{j= 2}^lh_{i_0}c_j\otimes\lambda_{i_0j}- h_{i_0}c_1\otimes(\sum_{j= 2}^l\lambda_{i_0j})= \sum_{j= 2}^lh_{i_0}(c_j- c_1)\otimes\lambda_{i_0j}$. Suppose $h_{i_0}c_j= \widetilde{c_j}h_{i_0}, h_{i_0}c_1= \widetilde{c_1}h_{i_0}$ for some $\widetilde{c_j}, \widetilde{c_1}\in L$, then 
\begin{equation*}
h_{i_0}(c_j- c_1)\otimes\lambda_{i_0j}= \widetilde{c_j}\widetilde{c_1}^{-1}\widetilde{c_1}h_{i_0}\otimes\lambda_{i_0j}- \widetilde{c_1}h_{i_0}\otimes\lambda_{i_0j}\in R.
\end{equation*}
So $\ker(\Phi)= R$ and we complete the proof of the claim.

$G$ is $p$-nilpotent implies $H/ L$ is a $p$-group. Since $X\mid \Ind_{PL}^H(k)$ and $\Ind_{PL}^H(k)$ is indecomposable by the Green's indecomposability theorem, we get that $X$ is indecomposable. 
\end{proof}

\section*{Acknowledgement}

The author would like to thank Zhenye Li for carefully reading and for helpful suggestions. The author thanks Shigeo Koshitani for carefully reading the manuscript and alerting the author that the citation of a withdrawn preprint needed clarification. 

\bibliography{reference}
\bibliographystyle{abbrv}

\end{document}